\documentclass{article}

\usepackage{amsmath,amsthm,amssymb,amsfonts}
\usepackage{verbatim}
\usepackage{array}

\usepackage{hyperref}

\usepackage{tikz}
\usepackage{tkz-berge, tkz-graph}
\tikzset{vertex/.style={circle,draw,fill,inner sep=0pt,minimum size=1mm}}

\usepackage{stmaryrd} 
\usepackage{hyperref}

\usepackage[colorinlistoftodos]{todonotes}

\usepackage{tikz}
\usepackage{tkz-berge, tkz-graph}

\tikzset{vertex/.style={circle,draw,fill,inner sep=0pt,minimum size=1mm}}

\theoremstyle{plain}
\newtheorem{thm}{Theorem}
\newtheorem{lem}[thm]{Lemma}
\newtheorem{prop}[thm]{Proposition}

\newtheorem{cor}[thm]{Corollary}

\newtheorem{remark}[thm]{Remark}

\theoremstyle{definition}

\newtheorem{definition}[thm]{Definition}
\newtheorem{exl}[thm]{Example}

\numberwithin{thm}{section}

\newcommand{\adj}{\leftrightarrow}
\newcommand{\adjeq}{\leftrightarroweq}

\DeclareMathOperator{\id}{id}
\def\N{{\mathbb N}}
\def\R{{\mathbb R}}

\DeclareMathOperator{\Fix}{Fix}
\newcommand{\Z}{\mathbb{Z}}
  
\begin{document}
\title{Fixed Point Sets in Digital Topology, 2}
\author{Laurence Boxer
         \thanks{
    Department of Computer and Information Sciences,
    Niagara University,
    Niagara University, NY 14109, USA;
    and Department of Computer Science and Engineering,
    State University of New York at Buffalo
    email: boxer@niagara.edu
    }
}
\date{}
\maketitle

\begin{abstract}
We continue the work of~\cite{bs19a}, studying
properties of digital 
images determined by fixed point invariants.
We introduce pointed versions of invariants
that were introduced in~\cite{bs19a}. We 
introduce freezing sets and cold sets to show
how the existence of a fixed point set for
a continuous self-map restricts the map on
the complement of the fixed point set.
\end{abstract}

\section{Introduction}
As stated in~\cite{bs19a}:
\begin{quote}
Digital images are often used as
mathematical models of real-world objects.
A digital model of the notion of a continuous function,
borrowed from the study of topology, is
often useful for the study of  digital images.
However, a digital image is typically
a finite, discrete point set. Thus,
it is often necessary to study digital
images using methods not directly
derived from topology. In this paper,
we examine some properties of digital 
images concerned with the
fixed points of digitally
continuous functions; among
these properties are discrete measures
that are not natural analogues of
properties of subsets of $\R^n$.
\end{quote}

In~\cite{bs19a}, we studied
rigidity, pull indices, fixed
point spectra for digital
images and for digitally continuous functions, and related notions. In the current
work, we study pointed versions
of notions introduced in~\cite{bs19a}.
We also study such questions
as when a set of fixed points
$\Fix(f)$ determines that
$f$ is an identity function, or is 
``approximately" an identity function.

Some of the results in this paper were presented in~\cite{BxVigo}.

\section{Preliminaries}
Much of this section is quoted
or paraphrased 
from~\cite{bs19a}.

Let $\N$ denote the set of natural numbers; 
$\N^* = \{0\} \cup \N$, the
set of nonnegative integers;
and $\Z$, the set of integers. $\#X$ will be
used for the number of members of a set~$X$.

\subsection{Adjacencies}
A digital image is a pair $(X,\kappa)$ where
$X \subset \Z^n$ for some $n$ and $\kappa$ is
an adjacency on $X$. Thus, $(X,\kappa)$ is a graph
for which $X$ is the vertex set and $\kappa$ 
determines the edge set. Usually, $X$ is finite,
although there are papers that consider infinite $X$. Usually, adjacency reflects some type of
``closeness" in $\Z^n$ of the adjacent points.
When these ``usual" conditions are satisfied, one
may consider the digital image as a model of a
black-and-white ``real world" image in which
the black points (foreground) are represented by 
the members of $X$ and the white points 
(background) by members of $\Z^n \setminus \{X\}$.

We write $x \adj_{\kappa} y$, or $x \adj y$ when
$\kappa$ is understood or when it is unnecessary to
mention $\kappa$, to indicate that $x$ and
$y$ are $\kappa$-adjacent. Notations 
$x \adjeq_{\kappa} y$, or $x \adjeq y$ when
$\kappa$ is understood, indicate that $x$ and
$y$ are $\kappa$-adjacent or are equal.

The most commonly used adjacencies are the
$c_u$ adjacencies, defined as follows.
Let $X \subset \Z^n$ and let $u \in \Z$,
$1 \le u \le n$. Then for points
\[x=(x_1, \ldots, x_n) \neq (y_1,\ldots,y_n)=y\]
we have $x \adj_{c_u} y$ if and only if
\begin{itemize}
    \item for at most $u$ indices $i$ we have
          $|x_i - y_i| = 1$, and
    \item for all indices $j$, $|x_j - y_j| \neq 1$
          implies $x_j=y_j$.
\end{itemize}

The $c_u$-adjacencies are often denoted by the
number of adjacent points a point can have in the
adjacency. E.g.,
\begin{itemize}
\item in $\Z$, $c_1$-adjacency is 2-adjacency;
\item in $\Z^2$, $c_1$-adjacency is 4-adjacency and
      $c_2$-adjacency is 8-adjacency;
\item in $\Z^3$, $c_1$-adjacency is 8-adjacency,
      $c_2$-adjacency is 18-adjacency, and 
      $c_3$-adjacency is 26-adjacency.
\end{itemize}

The literature also contains several adjacencies
to exploit properties of Cartesian products
of digital images. These include the following.

\begin{definition}
\cite{Berge}
Let $(X,\kappa)$ and $(Y, \lambda)$ be digital
images. The {\em normal product adjacency} or
{\em strong adjacency} on $X \times Y$,
$NP(\kappa, \lambda)$, is defined as follows.
Given $x_0, x_1 \in X$, $y_0, y_1 \in Y$ such that
\[p_0=(x_0,y_0) \neq (x_1,y_1)=p_1, \]
we have $p_0 \adj_{NP(\kappa,\lambda)} p_1$ if
and only if one of the following is valid:
\begin{itemize}
    \item $x_0 \adj_{\kappa} x_1$ and $y_0=y_1$, or
    \item $x_0 = x_1$ and $y_0 \adj_{\lambda} y_1$,
          or
    \item $x_0 \adj_{\kappa} x_1$ and
          $y_0 \adj_{\lambda} y_1$.
\end{itemize}
\end{definition}

Building on the normal product adjacency, we have
the following.

\begin{definition}
{\rm \cite{BxNormal}}
Given $u, v \in \N$, $1 \le u \le v$, and digital
images $(X_i, \kappa_i)$, $1 \le i \le v$, let
$X = \Pi_{i=1}^v X_i$. The adjacency
$NP_u(\kappa_1, \ldots, \kappa_v)$ for $X$ is
defined as follows. Given $x_i, x_i' \in X_i$, let
\[p=(x_1, \ldots, x_v) \neq (x_1', \ldots, x_v')=q. \]
Then $p \adj_{NP_u(\kappa_1, \ldots, \kappa_v)} q$
if for at least 1 and at most $u$ indices $i$ we have
$x_i \adj_{\kappa_i} x_i'$ and for all other
indices $j$ we have $x_j = x_j'$.
\end{definition}

Notice $NP(\kappa, \lambda)= NP_2(\kappa, \lambda)$
\cite{BxNormal}.

Let $x \in (X,\kappa)$. We use the notations
\[ N(x) = N_{\kappa}(x) = \{ y \in X \, | \, y \adj_{\kappa} x \}
\]
and
\[ N^*(x) = N_{\kappa}^*(x) = N_{\kappa}(x) \cup \{x\}.
\]

\subsection{Digitally continuous functions}
We denote by $\id$ or $\id_X$ the
identity map $\id(x)=x$ for all $x \in X$.

\begin{definition}
{\rm \cite{Rosenfeld, Bx99}}
Let $(X,\kappa)$ and $(Y,\lambda)$ be digital
images. A function $f: X \to Y$ is 
{\em $(\kappa,\lambda)$-continuous}, or
{\em digitally continuous} when $\kappa$ and
$\lambda$ are understood, if for every
$\kappa$-connected subset $X'$ of $X$,
$f(X')$ is a $\lambda$-connected subset of $Y$.
If $(X,\kappa)=(Y,\lambda)$, we say a function
is {\em $\kappa$-continuous} to abbreviate
``$(\kappa,\kappa)$-continuous."
\end{definition}

\begin{thm}
{\rm \cite{Bx99}}
A function $f: X \to Y$ between digital images
$(X,\kappa)$ and $(Y,\lambda)$ is
$(\kappa,\lambda)$-continuous if and only if for
every $x,y \in X$, if $x \adj_{\kappa} y$ then
$f(x) \adjeq_{\lambda} f(y)$.
\end{thm}

\begin{thm}
\label{composition}
{\rm \cite{Bx99}}
Let $f: (X, \kappa) \to (Y, \lambda)$ and
$g: (Y, \lambda) \to (Z, \mu)$ be continuous 
functions between digital images. Then
$g \circ f: (X, \kappa) \to (Z, \mu)$ is continuous.
\end{thm}

It is common to use the term {\em path} with the following distinct but related
meanings. 
\begin{itemize}
    \item A path from $x$ to $y$ in a digital image $(X,\kappa)$ is a set
          $\{x_i\}_{i=0}^m \subset X$ such that $x_0=x$, $x_m=y$, and
          $x_i \adjeq_{\kappa} x_{i+1}$ for $i=0,1,\ldots,m-1$. If the $x_i$ are distinct,
          then $m$ is the {\em length} of this path.
    \item A path from $x$ to $y$ in a digital image $(X,\kappa)$ is a
          $(2,\kappa)$-continuous function $P: [0,m]_{\Z} \to X$ such that $P(0)=x$ and
          $P(m)=y$. Notice that in this usage, $\{P(0), \ldots, P(m)\}$ is a path in the previous
          sense.
\end{itemize}

\begin{definition}
{\rm (\cite{Bx99}; see also \cite{Khalimsky})}
\label{htpy-2nd-def}
Let $X$ and $Y$ be digital images.
Let $f,g: X \rightarrow Y$ be $(\kappa,\kappa')$-continuous functions.
Suppose there is a positive integer $m$ and a function
$h: X \times [0,m]_{\Z} \rightarrow Y$
such that

\begin{itemize}
\item for all $x \in X$, $h(x,0) = f(x)$ and $h(x,m) = g(x)$;
\item for all $x \in X$, the induced function
      $h_x: [0,m]_{\Z} \rightarrow Y$ defined by
          \[ h_x(t) ~=~ h(x,t) \mbox{ for all } t \in [0,m]_{\Z} \]
          is $(2,\kappa')-$continuous. That is, $h_x$ is a path in $Y$.
\item for all $t \in [0,m]_{\Z}$, the induced function
         $h_t: X \rightarrow Y$ defined by
          \[ h_t(x) ~=~ h(x,t) \mbox{ for all } x \in  X \]
          is $(\kappa,\kappa')-$continuous.
\end{itemize}
Then $h$ is a {\em digital $(\kappa,\kappa')-$homotopy between} $f$ and
$g$, and $f$ and $g$ are {\em digitally $(\kappa,\kappa')-$homotopic in} $Y$, denoted
$f \sim_{\kappa,\kappa'} g$ or $f \sim g$ when
$\kappa$ and $\kappa'$ are understood.
If $(X,\kappa)=(Y,\kappa')$, we say $f$ and $g$ are
{\em $\kappa$-homotopic} to abbreviate
``$(\kappa,\kappa)$-homotopic" and write
$f \sim_{\kappa} g$ to abbreviate
``$f \sim_{\kappa,\kappa} g$". If further
$h(x,t)=x$ for all $t \in [0,m]_{\Z}$, we say $h$
{\em holds $x$ fixed}.

If there exists $x_0 \in X$
such that 
$f(x_0)=g(x_0)=y_0 \in Y$
and $h(x_0,t)=y_0$ for all
$t \in [0,m]_{\Z}$, then
$h$ is a {\em pointed homotopy} and $f$ and $g$
are {\em pointed homotopic}~\cite{Bx99}.

If there exist continuous $f: (X,\kappa) \to (Y,\lambda)$
and $g: (Y,\lambda) \to (X,\kappa)$ such that
$g \circ f \sim_{\kappa,\kappa} \id_X$ and
$f \circ g \sim_{\lambda,\lambda} \id_Y$, then
$(X,\kappa)$ and $(Y,\lambda)$ are {\em homotopy equivalent}.

If there is a $\kappa$-homotopy between
$\id_X$ and a constant map, we say $X$ is {\em $\kappa$-contractible}, or just {\em contractible}
when $\kappa$ is understood.
\end{definition}

\begin{thm}
{\rm \cite{BxNormal}}
\label{prodContinuity}
Let $(X_i,\kappa_i)$ and $(Y_i,\lambda_i)$
be digital images, $1 \le i \le v$. 
Let $f_i: X_i \to Y_i$. Then the product map
$f: \prod_{i=1}^v X_i \to \prod_{i=1}^v Y_i$
defined by 
\[ f(x_1,\ldots,x_v) = (f_1(x_1), \ldots, f_v(x_v))
\]
for $x_i \in X_i$, is
$(NP_v(\kappa_1,\ldots,\kappa_v),NP_v(\lambda_1,\ldots,\lambda_v))$-continuous if and only if each
$f_i$ is $(\kappa_i,\lambda_i)$-continuous.
\end{thm}

\begin{definition}
Let $A \subset X$. A $\kappa$-continuous
function $r: X \to A$ is a {\em retraction}, and
{\em $A$ is a retract of $X$}, if $r(a)=a$ for
all $a \in A$. If such a map $r$ satisfies
$i \circ r \sim_{\kappa} \id_X$ where $i: A \to X$ is the
inclusion map, then $A$ is a {\em $\kappa$-deformation retract} of $X$.
\end{definition}

A function $f: (X,\kappa) \to (Y,\lambda)$ is
an {\em isomorphism} (called a {\em homeomorphism}
in~\cite{Bx94}) if $f$ is a continuous bijection
such that $f^{-1}$ is continuous.

We use the following notation. For a
digital image $(X,\kappa)$,
\[ C(X,\kappa) = \{f: X \to X \, | \,
   f \mbox{ is continuous}\}.
\]

Given $f \in C(X,\kappa)$, a point
$x \in X$ is a {\em fixed point of $f$} if
$f(x)=x$. We denote by $\Fix(f)$ the set
$\{x \in X \, | \, x 
   \mbox{ is a fixed point of } f \}$.
If $x \in X \setminus \Fix(f)$, we say $f$ {\em moves} $x$.

\section{Rigidity and reducibility}
A function $f:(X,\kappa) \to (Y,\lambda)$
is \emph{rigid}~\cite{bs19a} when no 
continuous map is homotopic to $f$ 
except $f$ itself.
When the identity map 
$\id: X \to X$ 
is rigid, we say $X$ is rigid~\cite{hmps}.
If $f: X \to Y$ with 
$f(x_0)=y_0$, then
$f$ is {\em pointed rigid}~\cite{hmps} if
no continuous map is
pointed homotopic to $f$ other
than $f$ itself. When the 
identity map 
$\id: (X,x_0) \to (X,x_0)$ is 
pointed rigid, we say $(X,x_0)$ is 
pointed rigid.

Rigid maps and digital images are
discussed in~\cite{hmps,bs19a}.

Clearly, a rigid map is pointed rigid, and a rigid digital image is pointed rigid.
(Note these assertions may seem counterintuitive as, e.g., pointed homotopic
functions are homotopic, but the converse is not always true.) We show in the following that the
converses of these assertions 
are not generally true.

\begin{figure}
\includegraphics[height=2in]{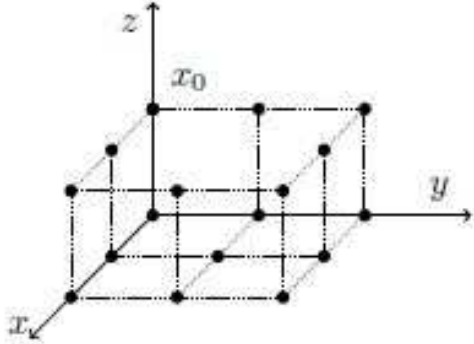}
\caption{(Figure~1 of~\cite{bs16}.) The image $X$ discussed in 
Example~\ref{contract-not-ptd}. The coordinates are ordered according to the axes in this figure.
}
\label{contract-not-ptd-fig}
\end{figure}
\begin{exl}
{\rm \cite{bs16}}
\label{contract-not-ptd}
Let $X= ([0,2]_{\Z}^2 \times [0,1]_{\Z}) \setminus \{(1,1,1)\}$. Let 
$x_0 = (0,0,1) \in X$.
See Figure~\ref{contract-not-ptd-fig}.
It was shown in~\cite{bs16} that
$X$ is 6-contractible (i.e.,
$c_1$-contractible) but 
$(X,x_0)$ is not
pointed 6-contractible. The
proof of the latter uses
an argument that is easily 
modified to show that any homotopy
of $\id_X$ that moves some point
must move $x_0$. It
follows that $\id_X$ is not
rigid but is $x_0$-pointed rigid, i.e., that
$X$ is not $c_1$-rigid but 
$(X,x_0)$ is $c_1$-pointed rigid.
\end{exl}

\begin{definition}
{\rm \cite{hmps}}
A finite image $X$ is {\em reducible} if it is
homotopy equivalent to an image of fewer points. Otherwise,
we say $X$ is {\em irreducible}.
\end{definition}

\begin{lem}
{\rm \cite{hmps}}
\label{htpNonsurj}
A finite image $X$ is reducible if and 
only if $\id_X$ is homotopic to a nonsurjective map.
\end{lem}

Let $(X,\kappa)$ be reducible. By 
Lemma~\ref{htpNonsurj}, there exist $x\in X$ and
$f \in C(X,\kappa)$ such that $\id_X \simeq_{\kappa} f$
and $x \not \in f(X)$. We will call such a point
a {\em reduction point}.

In Lemma~\ref{reducingPt} below, we have changed the
notation of~\cite{hmps}, since the latter paper uses
the notation ``$N(x)$" for what we call ``$N^*(x)$" or
``$N_{\kappa}^*(x)$".

\begin{lem}
{\rm \cite{hmps}}
\label{reducingPt}
If there exist distinct $x, y \in X$ so 
that $N^*(x) \subset N^*(y)$, then $X$ is 
reducible. In particular, $x$ is a reduction point
of $X$, and $X \setminus \{x\}$ is 
a deformation retract of $X$.
\end{lem}

\begin{remark}
{\rm \cite{hmps}}
A finite rigid image is irreducible.
\end{remark}

\begin{thm}
\label{1-2-3-reducible}
Let $(X,c_2)$ be a digital image 
in $\Z^2$. Suppose there exists $x_0 \in X$
such that $N_{c_2}(x_0)$ is $c_2$-connected
and $\#N_{c_2}(x_0) \in \{1,2,3\}$. 
Then $(X,c_2)$ is reducible.
\end{thm}

\begin{proof}
We first show that in all cases, there exists $y \in N_{c_2}(x_0)$ such that
$N_{c_2}^*(x_0) \subset N_{c_2}^*(y)$.
\begin{enumerate}
    \item Suppose $\#N_{c_2}(x_0) = 1$. Then there exists $y \in X$ such that
          $\{y\} = N_{c_2}(x_0)$. Clearly, then, $N_{c_2}^*(x_0) \subset N_{c_2}^*(y)$.
    \item Suppose $\#N_{c_2}(x_0) = 2$. Then there exist distinct $y, y' \in X$ such that
          $\{y,y'\} = N_{c_2}(x_0)$, which by hypothesis is connected. Therefore,
          $\{x_0,y'\} \subset N_{c_2}(y)$, so $N_{c_2}^*(x_0) \subset N_{c_2}^*(y)$.
    \item Suppose $\#N_{c_2}(x_0) = 3$. Then there exist distinct $y, y_0, y_1 \in X$ such that
          $\{y,y_0, y_1\} = N_{c_2}(x_0)$, which by hypothesis is connected. Therefore, one of
          the members of $N_{c_2}(x_0)$, say, $y$, is adjacent to the other two. Thus,
          $\{x_0,y_0, y_1\} \subset N_{c_2}(y)$, so $N_{c_2}^*(x_0) \subset N_{c_2}^*(y)$.
\end{enumerate}
Since in all cases we have $N_{c_2}^*(x_0) \subset N_{c_2}^*(y)$, the assertion 
follows from Lemma~\ref{reducingPt}.
\end{proof}

\begin{remark}
If instead we use the $c_1$-adjacency, the analog of the previous theorem
is simpler, since if $(X,c_1)$ is a digital image in $\Z^2$ and $x_0 \in X$
such that $N_{c_1}(x_0)$ is nonempty and $c_1$-connected, then
$\#N_{c_1}(x_0)=1$. This case is similar to the case $\#N_{c_2}(x_0)=1$ of
Theorem~\ref{1-2-3-reducible} above, so $(X,c_1)$ is reducible.
\end{remark}

\section{Pointed homotopy fixed point spectrum}
In this section, we define
pointed versions of the
{\em homotopy fixed point spectrum} of $f \in C(X,\kappa)$
and the {\em fixed point spectrum} of a digital image
$(X,\kappa)$.

\begin{definition}
Let $(X,\kappa)$ be a digital image.
\begin{itemize}
    \item {\rm ~\cite{bs19a}}
    Given $f \in C(X,\kappa)$,
the {\em homotopy fixed point spectrum of} $f$ is
\[ S(f) = \{\#\Fix(g) \, | \,
   g \sim_{\kappa} f\}.
\]
\item Given $f \in C(X,\kappa)$
      and $x_0 \in \Fix(f)$,
      the {\em pointed homotopy fixed point spectrum of} $f$ is
\[ S(f,x_0) = \{\#\Fix(g) \, | \,
   g \sim_{\kappa} f \mbox{  holding $x_0$ fixed}\}.
\]
\end{itemize}
\end{definition}

\begin{definition}
Let $(X,\kappa)$ be a digital image.
\begin{itemize}
    \item {\rm \cite{bs19a}} The {\em fixed point spectrum} of $(X,\kappa)$ is
\[ F(X) = F(X,\kappa) =
   \{\#\Fix(f) \, | \, f \in C(X,\kappa)\}.
\]
\item Given $x_0 \in X$, the
{\em pointed fixed point spectrum} of $(X,\kappa,x_0)$ is
\[ F(X,x_0) = F(X,\kappa,x_0) =
   \{\#\Fix(f) \, | \, f \in C(X,\kappa), x_0 \in \Fix(f)\}.
\]
\end{itemize}
\end{definition}

\begin{thm}
{\rm \cite{bs19a}}
\label{retractF}
Let $A$ be a retract of $(X,\kappa)$. Then
$F(A) \subseteq F(X)$.
\end{thm}

The argument used to prove
Theorem~\ref{retractF} is easily
modified to yield the following.

\begin{thm}
\label{retractFptd}
Let $(A,\kappa, x_0)$ be a retract of $(X,\kappa, x_0)$. Then
$F(A,\kappa, x_0) \subseteq F(X,\kappa, x_0)$.
\end{thm}

\begin{thm}
{\rm \cite{bs19a}}
\label{rectangleSpectrum}
Let $X = [1,a]_{\Z} \times [1,b]_{\Z}$. Let $\kappa \in \{c_1,c_2\}$. Then
\[ S(\id_X,\kappa) = F(X,\kappa) = \{i\}_{i=0}^{ab}.
\]
\end{thm}

\begin{exl}
Consider the pointed digital image
$(X,c_1,x_0)$ of Example~\ref{contract-not-ptd}.
Since $f \in C(X,c_1)$ and
$x_0 \in \Fix(f)$ imply
$f = \id_X$, 
\[ S(\id_X,c_1,x_0)=\{\#X\}=\{17\}.
\]
However, $(X,c_1)$ is not rigid. It is easily seen that
there is a $c_1$-deformation retraction of $X$ to
$\{(x,y,0) \in X\}$, which is isomorphic to $[1,3]_{\Z}^2$.
It follows from Theorem~\ref{retractF} and
Theorem~\ref{rectangleSpectrum}
that $\{i\}_{i=0}^9 \subset S(\id_X)$. Since every $f \in C(X,c_1)$ such that
$f \simeq_{c_1} \id_X$ and $f \neq \id_X$ moves every
point $q$ of $X$ such that $p_3(q)=1$, it follows easily that
\[ S(\id_X,c_1)=F(X,c_1) =
\{0,1,2,3,4,5,6,7,8,9,17\}.
\]
\end{exl}

\section{Freezing sets}
In this section, we consider
subsets of $\Fix(f)$ that determine
that $f \in C(X,\kappa)$ must be
the identity function $\id_X$.
Interesting questions include what
properties such sets have, and how
small they can be.

In classical topology, given a connected set
$X \subset \R^n$ and a continuous self-map
$f$ on $X$, knowledge of a finite subset $A$
of the fixed points of $f$ rarely tells us
much about the behavior of $f$ on 
$X \setminus A$. By contrast, we see in this
section that knowledge of a subset of the
fixed points of a continuous self-map $f$ on a
digital image can completely characterize $f$ as
an identity map.

\subsection{Definition and basic properties}
\begin{definition}
\label{freezeDef}
Let $(X,\kappa)$ be a
digital image. We say
$A \subset X$ is a 
{\em freezing set for $X$}
if given $g \in C(X,\kappa)$,
$A \subset \Fix(g)$ implies
$g=\id_X$.
\end{definition}

\begin{thm}
\label{freezeAndFix}
Let $(X,\kappa)$ be a digital image. Let $A \subset X$.
The following are equivalent.
\begin{enumerate}
    \item $A$ is a freezing set for $X$.
    \item $\id_X$ is the unique extension of $\id_A$ to a member of $C(X,\kappa)$.
    \item For every isomorphism $F: X \to (Y,\lambda)$, if $g: X \to Y$ is 
          $(\kappa,\lambda)$-continuous and $F|_A = g|_A$, then $g=F$.
    \item Any continuous $g:A \to Y$ has at most one extension to an
          isomorphism $\bar g:X \to Y$.
\end{enumerate}
\end{thm}

\begin{proof}
1) $\Leftrightarrow$ 2): This follows from Definition~\ref{freezeDef}.

1) $\Rightarrow$ 3): Suppose $A$ is a freezing set for $X$. 
Let $F: X \to Y$ be a $(\kappa,\lambda)$-isomorphism.
Let $g: X \to Y$ be $(\kappa,\lambda)$-continuous,
such that $g|_A = F|_A$. Then
\[F^{-1} \circ g|_A = F^{-1} \circ F|_A = \id_X|_A = \id_A.
\]
Since the composition of digitally continuous functions
is continuous, it follows by hypothesis that
$F^{-1} \circ g = \id_X$, and therefore that
\[ g = F \circ (F^{-1} \circ g) =
   F \circ \id_X = F.
\]

3) $\Rightarrow$ 1): Suppose for every isomorphism $F: X \to (Y,\lambda)$, if
$g: X \to Y$ is $(\kappa,\lambda)$-continuous
and $F|_A = g|_A$, then $g=F$.
For $g \in C(X,\kappa)$, $A \subset \Fix(g)$ implies
$g|_A = \id_X|_A$, so since $\id_X$ is an isomorphism, $g=\id_X$.

3) $\Rightarrow$ 4): This is elementary.

4) $\Rightarrow$ 2): This follows by taking $g$ to be the inclusion of $A$ into $X$,
which extends to $\id_X$.
\end{proof}

Freezing sets are topological invariants in the sense of the following.

\begin{thm}
\label{freezeInvariant}
Let $A$ be a freezing set for the digital image $(X,\kappa)$ and let
$F: (X,\kappa) \to (Y,\lambda)$ be an isomorphism. Then $F(A)$ is
a freezing set for $(Y,\lambda)$.
\end{thm}

\begin{proof}
Let $g \in C(Y,\lambda)$
such that $g|_{F(A)} = \id_Y|_{F(A)}$. Then
\[ g \circ F|_A = 
g|_{F(A)} \circ F|_A =
   \id_Y|_{F(A)} \circ F|_A
  = F|_A.
\]
By Theorem~\ref{freezeAndFix}, $g \circ F = F$. Thus
\[ g = (g \circ F) \circ F^{-1} = F \circ F^{-1} = \id_Y.
\]
By Definition~\ref{freezeDef}, $F(A)$ is a freezing set for $(Y,\lambda)$.
\end{proof}

We will use the following.

\begin{prop}
{\rm \cite{bs19a}}
\label{uniqueShortest}
Let $(X,\kappa)$ be a digital
image and $f \in C(X,\kappa)$.
Suppose $x,x' \in \Fix(f)$ are
such that there is a unique
shortest $\kappa$-path $P$ in $X$ 
from $x$ to $x'$. Then
$P \subset \Fix(f)$.
\end{prop}

Let $p_i: \Z^n \to \Z$ be the projection
to the $i^{th}$ coordinate:
$p_i(z_1,\ldots,z_n) = z_i$.

The following assertion can
be interpreted to say that
in a $c_u$-adjacency,
a continuous function that
moves a point~$p$ also moves
a point that is ``behind"
$p$. E.g., in $\Z^2$, if $q$ and $q'$ are
$c_1$- or $c_2$-adjacent with $q$
left, right, above, or below $q'$, and a
continuous function $f$ moves $q$ to the left,
right, higher, or lower, respectively, then
$f$ also moves $q'$ to the left,
right, higher, or lower, respectively.

\begin{lem}
\label{c1pulling}
Let $(X,c_u)\subset \Z^n$ be a digital image, 
$1 \le u \le n$. Let $q, q' \in X$ be such that
$q \adj_{c_u} q'$.
Let $f \in C(X,c_u)$.
\begin{enumerate}
    \item If $p_i(f(q)) > p_i(q) > p_i(q')$
          then $p_i(f(q')) > p_i(q')$.
    \item If $p_i(f(q)) < p_i(q) < p_i(q')$
          then $p_i(f(q')) < p_i(q')$.
\end{enumerate}
\end{lem}

\begin{proof}
\begin{enumerate}
    \item Suppose $p_i(f(q)) > p_i(q) > p_i(q')$. Since $q \adj_{c_u} q'$,
    if $p_i(q)=m$ then $p_i(q')=m-1$.
    Then $p_i(f(q))>m$. By continuity of
    $f$, we must have $f(q') \adjeq_{c_u} f(q)$, so $p_i(f(q')) \ge m > p_i(q')$.
    \item This case is proven similarly.
\end{enumerate}
\end{proof}

\begin{thm}
\label{retractNoFreeze}
Let $(X,\kappa)$ be a digital image. Let
$X'$ be a proper subset of $X$ that is a
retract of $X$. Then $X'$ does not contain
a freezing set for $(X,\kappa)$.
\end{thm}

\begin{proof}
Let $r: X \to X'$ be a retraction. Then
$f = i \circ r \in C(X,\kappa)$, where 
$i: X' \to X$ is the inclusion map.
Then $f|_{X'} = \id_{X'}$, but $f \neq \id_X$.
The assertion follows.
\end{proof}

\begin{cor}
\label{reduxPtAndFreeze}
Let $(X,\kappa)$ be a reducible digital image.
Let $x$ be a reduction point for $X$. Let $A$ be
a freezing set for $X$. Then $x \in A$.
\end{cor}

\begin{proof}
Since $x$ is a reduction point for $X$, by Lemma~\ref{reducingPt},
there is a retraction $r: X \to X \setminus \{x\}$. It follows that
$X \setminus \{x\}$ does not contain a freezing set for $(X,\kappa)$.
\end{proof}

\begin{prop}
\label{3nbr}
Let $(X,c_2)$ be a connected digital image in~$\Z^2$. Suppose $x_0 \in X$ is such
that $N_{c_2}(x_0)$ is connected and $\#N_{c_2}(x_0) \in \{1,2,3\}$. If~$A$
is a freezing set for $(X,c_2)$, then $x_0 \in A$.
\end{prop}

\begin{proof}
By the proof of Theorem~\ref{1-2-3-reducible}, we can use Lemma~\ref{reducingPt}
to conclude that $x_0$ is a reduction point. The assertion follows from Corollary~\ref{reduxPtAndFreeze}.
\end{proof}

Proposition~\ref{3nbr} cannot in general
be extended to permit $\#N_{c_2}(x_0) = 4$,
as shown in the following.

\begin{exl}
\label{irredAtx0}
Let $X = \{x_i\}_{i=0}^4 \subset \Z^2$, where 
\[x_0=(0,0),~x_1=(0,-1),~x_2=(1,0),~x_3=(0,1),~
  x_4=(-1,1).
\]
See Figure~2.
Then $N_{c_2}(x_0)$ is $c_2$-connected and
$\#N_{c_2}(x_0) = 4$. It is easily seen that
$X \setminus \{x_0\}$ is a freezing set for
$(X,c_2)$.
\end{exl}

\begin{figure}
\begin{center}
\includegraphics[height=1.5in]{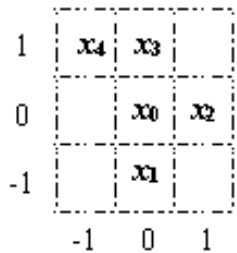}
\end{center}
\label{irreducibleAtx0}
\caption{Illustration for 
Example~\ref{irredAtx0} 
}
\end{figure}

\subsection{Boundaries and freezing sets}
For any digital image $(X,\kappa)$,
clearly $X$ is a freezing set. An
interesting question is how small
$A \subset X$ can be for $A$ to be
a freezing set for $X$. We say a
freezing set $A$ is {\em minimal}
if no proper subset of $A$ is a
freezing set for~$X$.

\begin{definition}
Let $X \subset \Z^n$.
\begin{itemize}
    \item The
{\em boundary of} $X$
{\rm \cite{RosenfeldMAA}} is
\[Bd(X) = \{x \in X \, | \mbox{ there exists } y \in \Z^n \setminus X \mbox{ such that } y \adj_{c_1} x\}.
\]
\item The {\em interior of} $X$
is $int(X) = X \setminus Bd(X)$.
\end{itemize}
\end{definition}

\begin{prop}
\label{intervalFreeze}
Let $[a,b]_{\Z} \subset [c,d]_{\Z}$ and
let $f: [a,b]_{\Z} \to [c,d]_{\Z}$ be
$c_1$-continuous.
\begin{itemize}
\item If $\{a,b\} \subset \Fix(f)$, then
      $[a,b]_{\Z} \subset \Fix(f)$.
\item $Bd([a,b]_{\Z}) = \{a,b\}$ is a minimal 
freezing set for $[a,b]_{\Z}$.
\end{itemize}
\end{prop}

\begin{proof}
If $[a,b]_{\Z} \neq \Fix(f)$, 
then we have at least one of the following:
\begin{itemize}
    \item For some smallest $t_0$ satisfying
          $a < t_0 < b$, $f(t_0) > t_0$. But
          then $f(t_0 - 1) \le t_0-1$, so
          $f(t_0 - 1) \not \adjeq_{c_1} f(t_0)$,
          contrary to the continuity of $f$.
    \item For some largest $t_1$ satisfying
          $a < t_1 < b$, $f(t_1) < t_1$. But
          then $f(t_1 + 1) \ge t_1+1$, so
          $f(t_1 + 1) \not \adjeq_{c_1} f(t_1)$,
          contrary to the continuity of $f$.
\end{itemize}
It follows that $f|_{[a,b]_{\Z}}$ is an
inclusion function, as asserted.

By taking $[c,d]_{\Z} = [a,b]_{\Z}$ and
considering all $f \in C([a,b]_{\Z}, c_1)$
such that $\{a,b\} \subset \Fix(f)$, we
conclude that $\{a,b\}$
is a freezing set for $[a,b]_{\Z}$.

To establish minimality, observe that all
proper subsets $B$ of $\{a,b\}$ allow
constant functions $c$ that are 
$c_1$-continuous non-identities 
with $c|_B = \id_B$.
\end{proof}

\begin{prop}
\label{bdAndInt}
Let $X \subset \Z^n$ be finite. Let $1 \le u \le n$.
Let $A \subset X$. Let $f \in C(X,c_u)$. If
$Bd(A) \subset \Fix(f)$, then $A \subset \Fix(f)$.
\end{prop}

\begin{proof}
By hypothesis, it suffices to
show $int(A) \subset \Fix(f)$.
Let $x = (x_1,\ldots,x_n) \in int(A)$.
Suppose, in order to obtain
a contradiction,
$x \not \in \Fix(f)$.
Then for some index~$j$,
\begin{equation}
\label{xNotFixed}
p_j(f(x)) \neq x_j.
\end{equation}
Since $X$ is finite,
there exists a path
$P=\{y_i = (x_1, \ldots, x_{j-1},a_i, x_{j+1} \ldots,x_n)\}_{i=1}^m$
in $X$ such
that $a_1 < x_j < a_m$ and $a_{i+1}=a_i + 1$;
$y_1,y_m \in Bd(A)$; and
$\{y_i\}_{i=2}^{m-1} \subset int(A)$. Note $x \in P$.
Now,~(\ref{xNotFixed})
implies either 
$p_j(f(x)) < x_j$ or
$p_j(f(x)) > x_j$.
If the former, then
by Lemma~\ref{c1pulling}, $y_m \not \in \Fix(f)$;
and if the latter, then
by Lemma~\ref{c1pulling}, $y_1 \not \in \Fix(f)$;
so in either case, we
have a contradiction.
We conclude that 
$x \in \Fix(f)$. The
assertion follows.
\end{proof}

\begin{thm}
\label{bdFreezes}
Let $X \subset \Z^n$ be finite. Then 
for $1 \le u \le n$, $Bd(X)$ is 
a freezing set for $(X,c_u)$.
\end{thm}

\begin{proof}
The assertion follows from
Proposition~\ref{bdAndInt}.
\end{proof}

Without the finiteness condition 
used in Proposition~\ref{bdAndInt}
and in Theorem~\ref{bdFreezes}, 
the assertions would
be false, as shown in the following.

\begin{exl}
Let $X = \{(x,y) \in \Z^2 \, | \, y \ge 0\}$. Consider
the function $f: X \to X$ defined by
\[ f(x,y) = \left \{ \begin{array}{ll}
    (x,0) & \mbox{if } y=0; \\
    (x+1,y) & \mbox{if } y>0.
\end{array} \right .
\]
Then $f \in C(X,c_2)$, $Bd(X) = \Z \times \{0\}$,
and $f|_{Bd(X)} = \id_{Bd(X)}$,
but $X \not \subset \Fix(f)$, so $Bd(X)$ is not a $c_2$-freezing
set for $X$.
\end{exl}

\subsection{Digital cubes and $c_1$}
In this section, we consider freezing sets for digital cubes using
the $c_1$ adjacency.

\begin{thm}
\label{corners-min}
Let $X = \Pi_{i=1}^n [0,m_i]_{\Z}$.
Let $A = \Pi_{i=1}^n \{0,m_i\}$.
\begin{itemize}
\item Let $Y = \Pi_{i=1}^n [a_i,b_i]_{\Z}$ be
      such that $[0,m_i] \subset [a_i,b_i]_{\Z}$ for all $i$. Let $f: X \to Y$ be
      $c_1$-continuous. If $A \subset \Fix(f)$, then $X \subset \Fix(f)$.
\item $A$ is a freezing set for $(X,c_1)$; minimal for $n \in \{1,2\}$.
\end{itemize}
\end{thm}

\begin{proof}
The first assertion has been established for 
$n=1$ at Proposition~\ref{intervalFreeze}.
We can regard this as a base case for an argument
based on induction on $n$, and we now assume 
the assertion is established for 
$n \le k$ where $k \ge 1$.

Now suppose $n=k+1$ and $f: X \to Y$ is $c_1$-continuous with $A \subset \Fix(f)$. Let
\[ X_0 = \Pi_{i=1}^k [0,m_i]_{\Z} \times \{0\}, ~~~ X_1 = \Pi_{i=1}^k [0,m_i]_{\Z} \times \{m_{k+1}\}.
\]
We have that $f|_{X_0}$ and $f|_{X_1}$ are $c_1$-continuous,
$A \cap X_0 \subset \Fix(f|_{X_0})$, and $A \cap X_1 \subset \Fix(f|_{X_1})$.
Since $X_0$ and $X_1$ are isomorphic to $k$-dimensional digital cubes,
by Theorem~\ref{freezeInvariant} and the inductive hypothesis, we have
\[ \left ( \Pi_{i=1}^k [0,m_i]_{\Z} \times \{0\} \right ) \cup
   \left ( \Pi_{i=1}^k [0,m_i]_{\Z} \times \{m_n\} \right )
   \subset \Fix(f).
\]
Then given $x = (x_1,\ldots,x_n) \in X$,
$x$ is a member of the unique shortest $c_1$-path
$\{(x_1,x_2,\ldots,x_k,t)\}_{t=0}^{m_1}$ from
$(x_1,x_2,\ldots,x_k,0) \in A$ to $(x_1,x_2,\ldots,x_k,m_n) \in A$.
By Proposition~\ref{uniqueShortest}, $x \in \Fix(f)$.
Since $x$ was taken arbitrarily, this completes the induction proof that 
$X \subset \Fix(f)$.

By taking $Y=X$ and applying the above to
all $f \in C(X,c_1)$ such that 
$A \subset \Fix(f)$, we conclude that
$A$ is a freezing set for $(X,c_1)$.

Minimality of $A$ for $n=1$ was established at
Proposition~\ref{intervalFreeze}. To show
minimality of $A$ for $n=2$, consider a
proper subset $A'$ of $A$. Without loss of
generality, $(0, 0) \in A \setminus A'$, $m_1 > 0$, and $m_2>0$. For
$x \in X$, let $g: X \to X$ be the function
\[ g(x)= \left \{ \begin{array}{ll}
     x & \mbox{if } x \neq (0,0); \\
     (1,1) &  \mbox{if } x = (0,0).
   \end{array} \right .
\]
Suppose $y \in X$ is such that $y \adj_{c_1} (0,0)$. Then $y=(1,0)$ or
$y = (0,1)$, hence
\[ g(y)=y \adj_{c_1} (1,1) =g(0,0).
\]
Thus $g \in C(X,c_1)$, $A' \subset \Fix(g)$,
and $g \neq \id_X$. Therefore, $A'$ is not
a freezing set for $(X,c_1)$, so $A$ is minimal.
\end{proof}

The minimality assertion of
Theorem~\ref{corners-min} does not extend
to $n=3$, as shown in the following.

\begin{exl}
\label{3cubeExl}
Let $X = [0,1]_{\Z}^3$. Let
\[ A = \{(0,0,0), (0,1,1), (1,0,1), (1,1,0)\}.
\]
See Figure~3. Then $A$ is a minimal freezing set for $(X,c_1)$.
\end{exl}

\begin{proof}
Note if $x \in X \setminus A$ then for each
index $i \in \{1,2,3\}$, $x$ is $c_1$-adjacent
to $y_i \in A$ such that $x$ and $y_i$ differ
in the $i^{th}$ coordinate. Therefore, if
$f \in C(X,c_1)$ such that $f(x) \neq x$, then
$c_1$-continuity requires that for some $i$ we 
have $f(y_i) \neq y_i$. It follows that $A$
is a freezing set for $(X,c_1)$.

Minimality is shown as follows. Let $A'$ be
a proper subset of $A$. Without 
loss of generality, $(0,0,0) \in A \setminus A'$.
Let $g: X \to X$ be the function (see 
Figure~3)
\[ g(x) = \left \{ \begin{array}{ll}
   (1,1,0) & \mbox{if } x=(0,0,0); \\
   (1,1,1) & \mbox{if } x=(0,0,1); \\
   x    & \mbox{otherwise.}
\end{array} \right .
\]
\begin{figure}
\begin{center}
\includegraphics[height=1.5in]{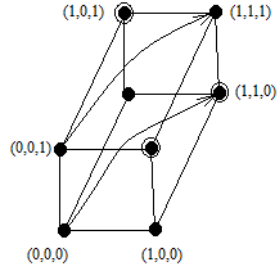}
\end{center}
\label{3cube}
\caption{The function $g$ in the proof of
Example~\ref{3cubeExl}. Members of
$A \setminus \{(0,0,0)\}$ are circled.
Straight line segments indicate $c_1$ adjacencies.
Curved arrows show the mapping for points in
$X \setminus \Fix(g)$.
}
\end{figure}
Then $g \in C(X,c_1)$, $g|_{A'} = \id_{A'}$,
and $g \neq \id_X$. Therefore, $A'$ is not a
freezing set for $(X,c_1)$.
\end{proof}

\subsection{Digital cubes and $c_n$}
In this section, we consider freezing sets for digital cubes in
$\Z^n$, using the $c_n$ adjacency.

\begin{thm}
\label{noProperSub}
Let $X = \prod_{i=1}^n[0,m_i]_{\Z}  \subset \Z^n$, where $m_i>1$ for all~$i$.
Then $Bd(X)$ is a minimal freezing set for $(X,c_n)$.
\end{thm}

\begin{proof}
That $Bd(X)$ is a freezing set for $(X,c_n)$ follows from Theorem~\ref{bdFreezes}.

To show $Bd(X)$ is a minimal freezing set, it suffices to show that if
$A$ is a proper subset of $Bd(X)$ then $A$ is not a freezing set for $(X,c_n)$.
We must show that there exists
\begin{equation}
\label{notFreeze}
f \in C(X,c_n) \mbox{ such that } f|_A = \id_A \mbox{ but }f \neq \id_X.
\end{equation}
By hypothesis, there exists
$y=(y_1,\ldots,y_n) \in Bd(X) \setminus A$.

Since $y \in Bd(X)$, for some index~$j$,
$y_j \in \{0,m_j\}$.
\begin{itemize}
    \item If $y_j=0$
the function $f: X \to X$ defined by \[f(y)=(y_1,\ldots,y_{j-1}, 1, y_{j+1},\ldots,y_n),~~~ f(x)=x \mbox{ for } x \neq y,
\]
satisfies~(\ref{notFreeze}).
\item If $y_j=m_j$
the function $f: X \to X$ defined by \[f(y)=(y_1,\ldots,y_{j-1}, m_j-1, y_{j+1},\ldots,y_n),~~~ f(x)=x \mbox{ for } x \neq y,
\]
satisfies~(\ref{notFreeze}).
\end{itemize}
The assertion follows.
\end{proof}

\subsection{Freezing sets and the normal product adjacency}
In the following,
$p_j: \prod_{i=1}^v X_i \to X_j$ is the map
\[ p_j(x_1, \ldots, x_v) = x_j \mbox{ where } x_i \in X_i.
\]

\begin{thm}
\label{prodFreezing}
Let $(X_i,\kappa_i)$ be a
digital image, $i \in [1,v]_{\Z}$. Let 
$X = \prod_{i=1}^v X_i$.
Let $A \subset X$. Suppose
$A$ is a freezing set
for $(X,NP_v(\kappa_1,\ldots,\kappa_v))$. Then
for each $i \in [1,v]_{\Z}$,
$p_i(A)$ is a freezing set
for $(X_i,\kappa_i)$.
\end{thm}

\begin{proof}
Let $f_i \in C(X_i,\kappa_i)$.
Let $F: X \to X$ be
defined by
\[ F(x_1,\ldots,x_v) =
   (f_1(x_1), \ldots, f_v(x_v)).
\]
Then by Theorem~\ref{prodContinuity}, 
$F \in C(X,NP_v(\kappa_1,\ldots, \kappa_v))$.

Suppose for all 
$a = (a_1,\ldots,a_v) \in A$, $F(a)=a$,
hence $f_i(a_i)=a_i$ for all 
$a_i \in p_i(A)$.
Since $A$ is a freezing set of $X$, we
have that $F = \id_X$, and therefore,
$f_i=\id_{X_i}$. The assertion follows.
\end{proof}

\subsection{Cycles}
A {\em cycle} or {\em digital simple closed curve} of $n$ distinct points
is a digital image $(C_n,\kappa)$
with $C_n = \{x_i\}_{i=0}^{n-1}$
such that $x_i \adj_{\kappa} x_j$
if and only if
$j=i+1 \mod n$ or
$j=i-1 \mod n$.

Given indices $i<j$, there are two
distinct paths determined by $x_i$ and $x_j$ in $C_n$,
consisting of the sets $P_{i,j}=\{x_k\}_{k=i}^j$ and
$P_{i,j}'= C_n \setminus \{x_k\}_{k=i+1}^{j-1}$. 
If one of these has length less than $n/2$, it is
the {\em shorter path} from $p_i$ to $p_j$ and
the other is the {\em longer path}; otherwise,
both have length $n/2$, and each is a 
{\em shorter path} and a {\em longer path}
from $p_i$ to $p_j$.

In this section, we consider
minimal fixed point sets for
$f \in C(C_n)$ that force
$f$ to be an identity map.

\begin{thm}
\label{3ptsForCycles}
Let $n > 4$. Let $x_i,x_j,x_k$ be
distinct members of $C_n$ be such
that $C_n$ is a union of unique shorter paths
determined by these points.
Let $f \in C(C_n,\kappa)$.
Then $f=\id_{C_n}$ if and only if
$\{x_i,x_j,x_k\} \subset \Fix(f)$; i.e.,
$\{x_i,x_j,x_k\}$ is a freezing set for $C_n$.
Further, this freezing set is minimal.
\end{thm}

\begin{proof}
Clearly $f=\id_{C_n}$ implies
$\{x_i,x_j,x_k\} \subset \Fix(f)$.

Suppose $\{x_i,x_j,x_k\} \subset \Fix(f)$. By hypothesis, there
are unique shorter paths $P_0$ 
from $x_i$ to $x_j$, $P_1$ from $x_j$ to $x_k$, and
$P_2$ from $x_k$ to $x_i$, in $C_n$. By
Proposition~\ref{uniqueShortest},
each of $P_0$, $P_1$, and
$P_2$ is contained in $\Fix(f)$.
By hypothesis $C_n = P_0 \cup P_1 \cup P_2$,
so $f = \id_{C_n}$. Hence $\{x_i,x_j,x_k\}$ is
a freezing set.

For any distinct pair $x_i,x_j \in C_n$,
there is a non-identity continuous self-map on 
$C_n$ that takes a longer path determined by
$x_i$ and $x_j$ to a shorter path determined by
$x_i$ and $x_j$. Thus, $\{x_i,x_j\}$ is not
a freezing set for $C_n$, so the set
$\{x_i,x_j,x_k\}$ discussed above is minimal.
\end{proof}

\begin{remark}
{\rm
In Theorem~\ref{3ptsForCycles},
we need the assumption that
$n>4$, as there is a continuous
self-map~$f$ on $C_4$ with 3
fixed points such that
$f \neq \id_{C_4}$~\cite{bs19a}.
}
\end{remark}

\subsection{Wedges}
Let $(X,\kappa) \subset \Z^n$ be such that
$X = X_0 \cup X_1$, where
$X_0 \cap X_1 = \{x_0\}$; and
if $x \in X_0$, $y \in X_1$, and
$x \adj_{\kappa} y$, then
$x_0 \in \{x,y\}$. We
say $X$ is the
{\em wedge of $X_0$ and $X_1$},
denoted $X = X_0 \vee X_1$. We
say $x_0$ is the {\em wedge point}.

\begin{thm}
Let $A$ be a freezing set for
$(X,\kappa)$, where 
$X = X_0 \vee X_1 \subset \Z^n$, $\#X_0 > 1$,
and $\#X_1 > 1$. 
Let $X_0 \cap X_1 = \{x_0\}$.
Then $A$ must include points of
$X_0 \setminus \{x_0\}$ and
$X_1 \setminus \{x_0\}$.
\end{thm}

\begin{proof}
Otherwise, either $A \subset X_0$ 
or $A \subset X_1$.

Suppose $A \subset X_0$. Then
the function $f: X \to X$ given
by
\[ f(x) = \left \{ \begin{array}{ll}
    x & \mbox{if } x \in X_0; \\
    x_0 & \mbox{if } x \in X_1,
\end{array} \right .
\]
belongs to $C(X,\kappa)$ and
satisfies $f|_A = \id_A$, but
$f \neq \id_X$. Thus $A$ is not a
freezing set for $(X,\kappa)$.

The case $A \subset X_1$ is argued
similarly.
\end{proof}

\begin{exl}
The wedge of two digital intervals
is (isomorphic to) a digital interval. It follows from
Theorem~\ref{freezeInvariant} and
Proposition~\ref{intervalFreeze}
that a freezing set
for a wedge need not include the wedge point.
\end{exl}

\begin{thm}
Let $C_m$ and $C_n$ be cycles,
with $m>4$, $n>4$. Let
$x_0$ be the wedge point of
$X=C_m \vee C_n$. Let
$x_i,x_j \in C_m$ and 
$x_i',x_j' \in C_n$ be such that
$C_m$ is the union of unique 
shorter paths determined by
$x_i, x_j, x_0$ and $C_n$ is the union of unique 
shorter paths determined by
$x_k', x_p', x_0$. Then
$A=\{x_i, x_j, x_k', x_p'\}$
is a freezing set for
$X$.
\end{thm}

\begin{proof}
Let $f \in C(X,\kappa)$ be such
that $A \subset \Fix(f)$. Let
$P_0$ be the unique shorter path in
$C_m$ from $x_i$ to $x_j$; let
$P_1$ be the unique shorter path in
$C_m$ from $x_j$ to $x_0$; let
$P_2$ be the unique shorter path in
$C_m$ from $x_0$ to $x_i$; let
$P_0'$ be the unique shorter 
path in $C_n$ from $x_k'$ to $x_p'$; let
$P_1'$ be the unique shorter 
path in $C_n$ from $x_p'$ to $x_0$;
let $P_2'$ be the unique shorter 
path in $C_n$ from $x_0$ to $x_k'$.

By Proposition~\ref{uniqueShortest},
each of the following paths is
contained in $\Fix(f)$: $P_0$,
$P_1 \cup P_1'$ (from $x_j$ to $x_0$ to $x_p'$), $P_2 \cup P_2'$ (from $x_i$ to $x_0$ to $x_k'$), and
$P_0'$. Since
\[X = P_0 \cup (P_1 \cup P_1') \cup 
  (P_2 \cup P_2') \cup P_0' \subset
  \Fix(f),
\]
the assertion follows.
\end{proof}

\subsection{Trees}
A {\em tree} is an acyclic graph $(X,\kappa)$ that is connected, i.e., lacking any 
subgraph isomorphic to $C_n$ for $n>2$. The {\em degree} of a vertex $x$
in $X$ is the number of distinct vertices $y \in X$ such that $x \adj y$. 
A vertex of a tree may be designated as the {\em root}. 
We have the following.

\begin{lem}
\label{tree-lem}
{\rm \cite{bs19a}}
Let $(X, \kappa)$ be a digital image that is a tree in which the root vertex has
at least 2 child vertices. Then $f \in C(X, \kappa)$
implies $\Fix(f)$ is $\kappa$-connected.
\end{lem}

\begin{thm}
\label{tree}
Let $(X,\kappa)$ be a digital image such that
the graph $G=(X,\kappa)$ is a finite tree with $\#X > 1$.
Let $E$ be the set of vertices of $G$ that have degree 1. Then
$E$ is a minimal freezing set for $G$.
\end{thm}

\begin{proof}
First consider the case that each vertex has degree~1. Since $X$ is a tree, it
follows that $X = \{x_0,x_1\} = E$, and $E$ is a freezing set. $E$ must be
minimal, since $X$ admits constant functions that are identities on their
restrictions to proper subsets of $E$.

Otherwise, there exists $x_0 \in X$ such that $x_0$ has degree of at least~2
in $G$. This implies $\#X > 2$, and
since $G$ is finite and acyclic, $\#E > 0$. Since $G$ is acyclic, removal of any member
of $X \setminus E$ would disconnect $X$. If we take $x_0$ to be the root vertex,
it follows from Lemma~\ref{tree-lem} that $E$ is a freezing set.

Since $\#E > 0$, for any $y \in E$ there exists $y' \in X \setminus E$
such that $y' \adj y$. Then the function
$f: X \to X$ defined by
\[ f(x) = \left \{ \begin{array}{ll}
   y' & \mbox{if } x = y; \\
   x & \mbox{if } x \neq y,
\end{array} \right .
\]
satisfies $f \in C(X,\kappa)$, 
$f|_{E \setminus \{y\}} = \id_{E \setminus \{y\}}$, and
$f \neq \id_X$. Thus $E \setminus \{y\}$ is not a freezing set.
Since $y$ was arbitrarily chosen, $E$ is minimal.
\end{proof}

\section{$s$-Cold sets}
In this section, we generalize our focus from
fixed points to approximate fixed points and, more
generally, to points constrained in the amount they
can be moved by continuous self-maps 
in the presence of fixed point sets. 
We obtain some analogues of our previous
results for freezing sets.

\subsection{Definition and basic properties}
In the following, we use the path-length metric $d$ for
connected digital images $(X,\kappa)$,
defined~\cite{Han05} as
\[ d(x,y)= \min\{\ell \, | \, \ell
      \mbox{ is the length of a $\kappa$-path in $X$ from $x$ to $y$} \}.
\]
If $X$ is finite and $\kappa$-connected, the {\em diameter} of $(X,\kappa)$ is
\[ diam(X,\kappa) = \max\{d(x,y) \, | \, x,y \in X\}.
\]
We introduce the following generalization of a freezing set.

\begin{definition}
\label{s-cold-def}
Given $s \in \N^*$, we say $A \subset X$ is an
{\em $s$-cold set} for the connected digital image $(X,\kappa)$
if given $g \in C(X,\kappa)$ such that
$g|_A = \id_A$, then for all $x \in X$, $d(x,g(x)) \le s$.
A {\em cold set} is a 1-cold set.
\end{definition}

Note a 0-cold set is a freezing set.

\begin{thm}
\label{coldSetEquiv}
Let $(X, \kappa)$ be a connected digital image.
Let $A \subset X$. 
Then $A \subset Fix(g)$ is an $s$-cold set for 
$(X,\kappa)$ if and only 
if for every isomorphism $F: (X,\kappa) \to (Y,\lambda)$,
if $g: X \to Y$ is $(\kappa,\lambda)$-continuous and
$F|_A = g|_A$, then for all $x \in X$, 
$d(F(x),g(x)) \le s$.
\end{thm}

\begin{proof}
Suppose $A$ is an $s$-cold set for $(X,\kappa)$. Then
for all $f \in C(X,\kappa)$ such that $f|_A = \id_A$ and all $x \in X$,
we have $d(x,f(x)) \le s$. Let
$F: (X,\kappa) \to (Y,\lambda)$ be an isomorphism. Let
$g: X \to Y$ be $(\kappa,\lambda)$-continuous with
$F|_A = g|_A$. Then 
\[ \id_A = F^{-1} \circ F|_A = F^{-1} \circ g|_A.
\]
Let $x \in X$. Then $d(x,F^{-1} \circ g(x)) \le s$, i.e.,
there is a $\kappa$-path $P$ in $X$ of length at most $s$
from $x$ to $F^{-1} \circ g(x)$. Therefore, $F(P)$ is
a $\lambda$-path in $Y$ of length at most $s$
from $F(x)$ to $F \circ F^{-1} \circ g(x) = g(x)$, i.e.,
$d(F(x),g(x)) \le s$.

Suppose $A \subset X$ and for every isomorphism $F: (X,\kappa) \to (Y,\lambda)$,
if $g: X \to Y$ is $(\kappa,\lambda)$-continuous and
$F|_A = g|_A$, then for all $x \in X$, $d(F(x),g(x)) \le s$.
Let $f \in C(X,\kappa)$ with $f|_A = \id_A$. Since $\id_X$ is an isomorphism, 
for all $x \in X$, $d(x,f(x)) \le s$. Thus, $A$ is an $s$-cold set for $(X,\kappa)$.
\end{proof}

Given a digital image $(X,\kappa)$ and
$f \in C(X,\kappa)$, a point 
$x \in X$ is an
{\em almost fixed point of}~$f$~\cite{Rosenfeld} or
an {\em approximate fixed point of}~$f$~\cite{BEKLL}
if $f(x) \adjeq_{\kappa} x$.

\begin{remark}
The following are easily observed.
\begin{itemize}
    \item If $A \subset A' \subset X$ and
          $A$ is an $s$-cold set for
          $(X,\kappa)$, then $A'$ is an
          $s$-cold set for $(X,\kappa)$.
    \item $A$ is a cold set (i.e., a 1-cold set)
          for $(X,\kappa)$ if and only if given $f \in C(X,\kappa)$ 
          such that $f|_A = \id_A$, every $x \in X$ is an approximate fixed point of $f$.
    \item In a finite connected digital image $(X,\kappa)$, every nonempty
          subset of $X$ is a $diam(X)$-cold set.
    \item If $s_0 < s_1$ and $A$ is an $s_0$-cold set
          for $(X,\kappa)$, then $A$ is an $s_1$-cold set for $(X,\kappa)$.
\end{itemize}
\end{remark}

Note a freezing set is a cold set, but the
converse is not generally true, as shown in the following.

\begin{exl}
It follows from Definition~\ref{s-cold-def} that
$\{0\}$ is a cold set, but not a freezing set, for
$X=[0,1]_{\Z}$, since the constant function $g$ with
value 0 satisfies $g|_{\{0\}}=\id_{\{0\}}$,
and $g(1)=0 \adj_{c_1} 1$.
\end{exl}

$s$-cold sets are invariant in the sense of the following.

\begin{thm}
\label{s-cold-invariant}
Let $(X,\kappa)$ be a connected digital image, 
let $A$ be an $s$-cold set for $(X,\kappa)$,
and let $F: (X,\kappa) \to (Y,\lambda)$ be an
isomorphism. Then $F(A)$ is
an $s$-cold set for $(Y,\lambda)$.
\end{thm}

\begin{proof}
Let $f \in C(Y,\lambda)$ such that $f|_{F(A)} = \id_{F(A)}$. Then
\[ f \circ F|_A =   f|_{F(A)} \circ F|_A =
   \id_{F(A)} \circ F|_A   = F|_A.
\]
By Theorem~\ref{coldSetEquiv}, for all $x \in X$,
$d(f \circ F(x), F(x)) \le s$.
Substituting $y=F(x)$, we have that $y \in Y$ implies $d(f(y), y) \le s$.
By Definition~\ref{s-cold-def}, $F(A)$ is a cold set
for $(Y,\lambda)$.
\end{proof}

$A$ is a {\em $\kappa$-dominating set} 
(or a {\em dominating set} when
$\kappa$ is understood) for $(X,\kappa)$ if for
every $x \in X$ there exists $a \in A$ such that
$x \adjeq_{\kappa} a$~\cite{CandL}. This notion is somewhat
analogous to that of a dense set in a topological space, and the
following is somewhat analogous to the fact that in topological
spaces, a continuous function is uniquely determined by its
values on a dense subset of the domain.

\begin{thm}
Let $(X,\kappa)$ be a digital image and let
$A$ be $\kappa$-dominating in $X$. Then $A$
is $2$-cold in $(X,\kappa)$.
\end{thm}

\begin{proof}
Let $f \in C(X,\kappa)$ such that $f|_A = \id_A$.
Since $A$ is $\kappa$-dominating, for every
$x \in X$ there is an $a \in A$ such that
$x \adjeq a$. Then $f(x) \adjeq f(a)=a$. Thus,
we have the path $\{x,a,f(x)\} \subset X$
from $x$ to $f(x)$ of length at most 2.
The assertion follows.
\end{proof}

\begin{thm}
Let $(X,\kappa)$ be rigid. If $A$ is a cold set
for $X$, then $A$ is a freezing set for $X$.
\end{thm}

\begin{proof}
Let $f \in C(X,\kappa)$ be such that
$f|_A = \id_A$. Since $A$ is cold,
$f(x) \adjeq x$ for all $x \in X$. Therefore,
the map $H: X \times [0,1]_{\Z} \to X$
defined by $H(x,0)=x$, $H(x,1)=f(x)$, is a
homotopy. Since $X$ is rigid, $f = \id_X$. The
assertion follows.
\end{proof}

\subsection{Cold sets for cubes}
In this section, we consider cold sets for digital
cubes in $\Z^n$. Note the hypotheses
of Proposition~\ref{c2-coldSet} imply $A$ is
$c_1$- and $c_2$-dominating in $Bd(X)$.

\begin{prop}
\label{c2-coldSet}
Let $m,n \in \N$. Let
$X = [0,m]_{\Z} \times [0,n]_{\Z}$.
Let $A \subset Bd(X)$ be such that
no pair of $c_1$-adjacent members of $Bd(X)$ 
belong to $Bd(X) \setminus A$.
Then $A$ is a cold set for $(X,c_2)$. Further,
for all $f \in C(X,c_2)$, if $f|_A = \id_A$ then
$f|_{Int(X)} = \id|_{Int(X)}$.
\end{prop}

\begin{proof}
Let $x=(x_0,y_0) \in X$. Let $f \in C(X,c_2)$
such that $f|_A = \id_A$.
Consider the following.
\begin{itemize}
    \item If $x \in A$ then $f(x)=x$.
    \item If $x \in Bd(X) \setminus A$ then both of
          the $c_1$-neighbors of $x$ in $Bd(X)$
          belong to $A$. We will show 
          $f(x) \adjeq_{c_2} x$.
          
          Let $K = \{(0,0), (0,n), (m,0), (m,n) \} \subset Bd(X)$.
          \begin{itemize}
              \item For $x \in K$, consider the case $x=(0,0)$. Then 
          $\{(0,1), (1,0)\} \subset A$, so we must have 
          \[ f(x) \in N_{c_2}^*((0,1)) \cap N_{c_2}^*((1,0)) \subset N_{c_2}^*(x).
          \]
          For other $x \in K$, we similarly find
          $f(x) \adjeq_{c_2} x$.
              \item For $x \in Bd(X) \setminus K$, 
                   consider the case $x=(t,0)$.
                   For this case,
                   $\{(t-1,0), (t+1,0)\} \subset A$,
                   so 
                   \[(t-1,0) = f(t-1,0)
                   \adjeq_{c_2} f(x) \adjeq_{c_2}
                   f(t+1,0) = (t+1,0).
                   \]
                   Therefore, $f(x) \in \{x,(t,1)\}$, so $f(x) \adjeq_{c_2} x$.
                   
                   For other $x \in Bd(X) \setminus K$, we similarly find
          $f(x) \adjeq_{c_2} x$.
          \end{itemize}
    \item If $x \in Int(X)$, let
$L=\{(z,0)\}_{z=x_0-1}^{x_0+1}$ and
$U=\{(z,n)\}_{z=x_0-1}^{x_0+1}$.
We have
\[ L \cap A \neq \emptyset \neq U \cap A.
\]

\begin{figure}
\includegraphics[height=2.5in]{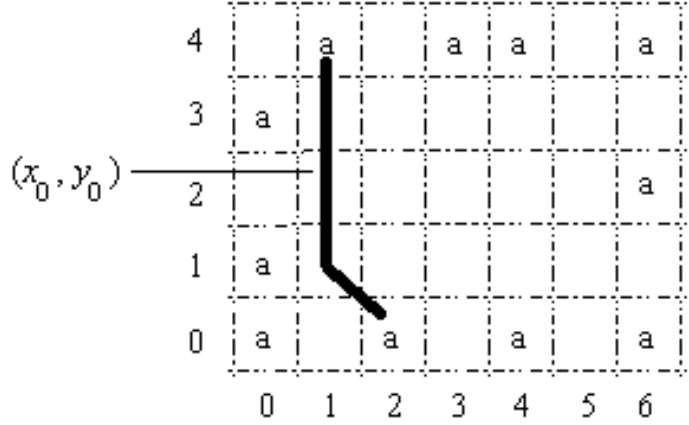}
\label{thePath}
\caption{Illustration of the proof of
Proposition~\ref{c2-coldSet} for the case
$(x_0,y_0) \in Int(X)$. 
$X = [0,6]_{\Z} \times [0,4]_{\Z}$. 
Members of the set $A \subset Bd(X)$ 
are marked ``a". Corner points 
such as $(0,4)$ need not belong to $A$; also, although we cannot have
$c_1$-adjacent members of $Bd(X)$
in $Bd(X) \setminus A$, we can have
$c_2$-adjacent members of $Bd(X)$
in $Bd(X) \setminus A$, e.g.,
$(5,4)$ and $(6,3)$. The heavy polygonal line illustrates a $c_2$-path~$P$ of length $n=4$:
$P(0)=q_L=(2,0)$, $P(1)=(1,1)$,
$P(2)=(x_0,y_0)=(1,2)$, $P(3)=(1,3)$,
$P(4)=q_U=(1,4)$.
}
\end{figure}
Since no pair of $c_1$-adjacent members of $Bd(X)$
belong to $Bd(X) \setminus A$, there exist
$q_L \in L \cap A$, $q_U \in U \cap A$ such that
\[ |p_1(q_L)-x_0| \le 1 \mbox{ and }
|p_1(q_U)-x_0| \le 1.
\]
Thus, there is an injective $c_2$-path 
$P: [0,n]_{\Z} \to X$ such that
$P([0,y_0)]_{\Z})$ runs from $q_L$ to $x$ and
$P([y_0,n]_{\Z})$ runs from~$x$
to $q_U$ (note since we use $c_2$-adjacency, there can be steps
of the path that change both 
coordinates - see 
Figure~4).
Therefore, $f \circ P$
is a path from $f(q_L)=q_L$ to $f(x)$
to $f(q_U)=q_U$, and
$p_2 \circ f \circ P$ is a path 
from $p_2(q_L)=0$
to $p_2(f(x))$ to $p_2(q_U)=n$.

If $y' = p_2(f(x)) > y_0$, then
$p_2 \circ f \circ P|_{[0,y_0]_{\Z}}$
is a $c_2$-path of length $y_0$ 
from 0 to $y'$, which is impossible.
Similarly, if $y' < y_0$,
then $p_2 \circ f \circ P|_{[y_0,n]_{\Z}}$
is a $c_2$-path of length $n-y_0$ 
from $y'$ to $n$, which is impossible.
Therefore, we must have 
\begin{equation}
\label{yBounds}
p_2 \circ f(x) = y_0.
\end{equation}

Similarly, by replacing the neighborhoods of the projections of $x$ on the
lower and upper edges of the cube, $L$ and $U$, by the neighborhoods of the projections of $x$ on the
the left and right edges of the cube, $L'=\{(0,z)\}_{z=y_0-1}^{y_0+1}$ and
$R=\{(m,z)\}_{z=y_0-1}^{y_0+1}$, and using an argument similar to that used to 
obtain~(\ref{yBounds}), we conclude that
\begin{equation}
    \label{xBounds}
    p_1 \circ f(x) = x_0.
\end{equation}
It follows from~(\ref{xBounds})
and~(\ref{yBounds}) that $f(x) = x$.
\end{itemize}
Thus, in all cases, $f(x) \adjeq_{c_2} x$, and
$f|_{Int(X)} = \id_{Int(X)}$.
\end{proof}

\begin{prop}
\label{c2-2coldSet}
Let $m,n \in \N$. Let
$X = [0,m]_{\Z} \times [0,n]_{\Z}$.
Let $A \subset Bd(X)$ be $c_1$-dominating
in $Bd(X)$.
Then $A$ is a 2-cold set for $(X,c_2)$. Further,
for all $f \in C(X,c_2)$, if $f|_A = \id_A$ then
$f|_{Int(X)} = \id|_{Int(X)}$.
\end{prop}

\begin{proof}
Our argument is similar to that of
Proposition~\ref{c2-coldSet}.
Let $x=(x_0,y_0) \in X$. Let $f \in C(X,c_2)$
such that $f|_A = \id_A$.
Consider the following.
\begin{itemize}
    \item If $x \in A$ then $f(x)=x$.
    \item If $x \in Bd(X) \setminus A$ then for some $a \in A$,
    $x \adjeq_{c_1} a$. Therefore,
    $f(x) \adjeq_{c_1} f(a)= a$. Thus,
    $\{x,a,f(x)\}$ is a path in $X$ from
    $x$ to $f(x)$ of length at most 2.
    \item If $x \in Int(X)$, then as
    in the proof of Proposition~\ref{c2-coldSet} we have
  that $f(x) = x$.
\end{itemize}
Thus, in all cases, $d(f(x), x) \le 2$, 
and $f|_{Int(X)} = \id_{Int(X)}$.
\end{proof}

An example of a 2-cold set~$A$ that is not
a 1-cold set, such that $A$ is as in
Proposition~\ref{c2-2coldSet}, is given
in the following.

\begin{exl}
Let $X = [0,2]_{\Z}^2$. Let 
\[ A = \{(0,2), (1,0), (2,2)\} \subset X.
\]
Then $A$ is $c_1$-dominating in $Bd(X)$,
so by Proposition~\ref{c2-2coldSet}, is
a 2-cold set for $(X,c_2)$. Let $f: X \to X$
be the function $f(0,0)=(2,0)$,
$f(0,1)=(1,1)$, and $f(x)=x$ for all
$x \in X \setminus \{(0,0),(0,1)\}$.
Then $f \in C(X,c_2)$ but
$d((0,0), f(0,0)) = 2$, so $A$ is not a
1-cold set.

\end{exl}

\begin{prop}
\label{notCold}
Let $X = \prod_{i=1}^n [0,m_i]_{\Z} \subset \Z^n$, 
where $m_i > 1$ for all~$i$.
Let $A \subset Bd(X)$ be such that
$A$ is not $c_n$-dominating in $Bd(X)$.
Then $A$ is not a cold set for $(X,c_n)$.
\end{prop}

\begin{proof}
By hypothesis, there exists
$y=(y_1,\ldots,y_n) \in Bd(X) \setminus A$ such that
$N(y,c_n) \cap A = \emptyset$.

Since $y \in Bd(X)$, for some index~$j$
we have $y_j \in \{0,m_j\}$. Let
$x=(x_1,\ldots, x_n) \in X$, for
$x_i \in [0,m_i]_{\Z}$.
\begin{itemize}
\item If $y_j=0$, let $f: X \to X$ 
be defined as follows.
\[ f(x) = \left \{ \begin{array}{cl}
    (x_1,\ldots, x_{j-1}, 2, x_{j+1}, \ldots,x_n) & \mbox{if } x=y; \\
    (x_1,\ldots, x_{j-1}, 1, x_{j+1}, \ldots,x_n) & \mbox{if } x \in N_{c_n}(y); \\
    x & \mbox{otherwise.}
    \end{array} \right .
\]
If $u,v \in X$, $u \adjeq_{c_n} v$, then
$u$ and $v$ differ by at most 1 in every 
coordinate. Consider the following
cases.
\begin{itemize}
    \item If $u=y$, then $v \in N_{c_n}(y)$, and clearly $f(u)$ and 
    $f(v)$ differ by at most 1 in every coordinate, hence are $c_n$-adjacent.
    Similarly if $v=y$.
    \item If $u,v \in N_{c_n}(y)$, then clearly $f(u)$ and $f(v)$ differ by 
    at most 1 in every coordinate, hence are $c_n$-adjacent.
    \item If $u \in N_{c_n}(y)$ and $v \not \in N_{c_n}(y)$, then 
    $p_j(u) \in \{0,1\}$, so $p_j(f(v)) = p_j(v) \in \{0,1,2\}$, and $p_j(f(u))=1$. It follows easily that
    $f(u)$ and $f(v)$ differ by at most 1 in every coordinate, hence are $c_n$-adjacent.
    Similarly if $v \in N_{c_n}(y)$ and $u \not \in N_{c_n}(y)$
    \item Otherwise, $\{u,v\} \cap N_{c_n}^*(y) = \emptyset$, so $f(u) = u \adjeq_{c_n} v=f(v)$.
\end{itemize}
Therefore, $f \in C(X,c_n)$.
\item If $y_j=m_j$, let $f: X \to X$ 
be defined by
\[ f(x) = \left \{ \begin{array}{cl}
    (x_1,\ldots, x_{j-1}, m_j-2, x_{j+1}, \ldots,x_n) & \mbox{if } x=y; \\
    (x_1,\ldots, x_{j-1}, m_j-1, x_{j+1}, \ldots,x_n) & \mbox{if } x \in N_{c_n}(y); \\
    x & \mbox{otherwise.}
    \end{array} \right .
\]
By an argument similar to that of the
case $y_j=0$, we conclude that
$f \in C(X,c_n)$.
\end{itemize}

Further, in both cases, $f|_A = \id_A$, and $f(y) \not \adjeq_{c_n} y$.
The assertion follows.
\end{proof}

\subsection{$s$-cold sets for rectangles in $\Z^2$}
The following generalizes the case $n=2$ of
Theorem~\ref{corners-min}.

\begin{prop}
Let $X = [-m,m]_{\Z} \times [-n,n]_{\Z} \subset \Z^2$, 
$s \in \N^*$, where $s \le \min\{m,n\}$. Let
\[ A = \{(-m+s, -n+s), (-m+s,n-s), (m-s, -n+s),
         (m-s,n-s)\}.
\]
Then $A$ is a $4s$-cold set for $(X,c_1)$.
\end{prop}

\begin{proof}
Let $f \in C(X,c_1)$ such that $f|_A = \id_A$.
Let 
\[ A' = [-m+s,m-s]_{\Z} \times [-n+s,n-s]_{\Z}.
\]
By Proposition~\ref{uniqueShortest},
$Bd(A') \subset \Fix(f)$. It follows from
Proposition~\ref{bdAndInt} that
$A' \subset \Fix(f)$.

Thus it remains to show that
$x \in X \setminus A'$ implies
$d(x,f(x)) \le 4s$. This is seen as follows.
For $x \in X \setminus A'$, there exists
a $c_1$-path $P$ of length at most $2s$ from 
$x$ to some $y \in Bd(A')$. Then
$f(P)$ is a $c_1$-path from $f(x)$ to $f(y)=y$
of length at most $2s$. Therefore,
$P \cup f(P)$ contains a path from $x$ to $y$ to
$f(x)$ of length at most $4s$. The assertion
follows.
\end{proof}

The following generalizes Proposition~\ref{c2-coldSet}.

\begin{prop}
Let $X = [-m,m]_{\Z} \times [-n,n]_{\Z} \subset \Z^2$, 
$s \in \N^*$, where $m-s \ge 0$, $n-s \ge 0$. Let
\[ A = [-m+s,m-s]_{\Z} \times [-n+s,n-s]_{\Z}
   \subset X.
\]
Let $A' \subset Bd(A)$ such that no pair of
$c_1$-adjacent members of $Bd(A)$ belongs to
$Bd(A) \setminus A'$.

Then $A'$ is a $2s$-cold set for
          $(X,c_2)$.

Further, if $f \in C(X,c_2)$ and 
$f|_{A'} = \id_{A'}$, then $f|_A = \id_A$.
\end{prop}

\begin{proof}
Let $f \in C(X,c_2)$ be such that
$f|_{A'} = \id_{A'}$. As in
the proof of Proposition~\ref{c2-coldSet},
$f|_A = \id_A$.

Now consider $x \in X \setminus A$. There is a $c_2$-path $P$ in $X$
          from $x$ to some $y \in A'$ of
          length at most $s$. Then $f(P)$
          is a $c_2$-path in $X$ from $f(x)$ to
          $f(y)=y$ of length at most $s$.
          Therefore, $P \cup f(P)$ contains
          a $c_2$-path in $X$ from $x$ to $y$
          to $f(x)$ of length at most $2s$. 
The assertion follows.
\end{proof}

\subsection{$s$-cold sets for Cartesian products}
We modify the proof of Theorem~\ref{prodFreezing} to
obtain the following.

\begin{thm}
\label{prodCold}
Let $(X_i,\kappa_i)$ be a
digital image, $i \in [1,v]_{\Z}$. 
Let $X = \prod_{i=1}^v X_i$.
Let $s \in \N^*$.
Let $A \subset X$. Suppose
$A$ is an $s$-cold set
for $(X,NP_v(\kappa_1,\ldots,\kappa_v))$. Then
for each $i \in [1,v]_{\Z}$,
$p_i(A)$ is an $s$-cold set
for $(X_i,\kappa_i)$.
\end{thm}

\begin{proof}
Let $f_i \in C(X_i,\kappa_i)$.
Let $F: X \to X$ be
defined by
\[ F(x_1,\ldots,x_v) =
   (f_1(x_1), \ldots, f_v(x_v)).
\]
Then by Theorem~\ref{prodContinuity},
$F \in C(X,NP_v(\kappa_1,\ldots, \kappa_v))$.

Suppose for all $i$, $a_i \in p_i(A)$,
we have $f_i(a_i)=a_i$. Note
this implies, for $a = (a_1,\ldots,a_v)$, that $F(a)=a$.
Since $a$ is an arbitrary member of
the $s$-set $A$ of $X$, we
have that $d(F(x), x) \le s$,
for all $x = (x_1,\ldots, x_v) \in X$, $x_i \in X_i$, and therefore,
$d(f_i(x_i), x_i) \le s$. The assertion follows.
\end{proof}

\subsection{$s$-cold sets for infinite digital images}
In this section, we obtain properties of
$s$-cold sets for some infinite digital images.

\begin{thm}
\label{infiniteFreeze}
Let $(\Z^n, c_u)$ be a digital image, 
$1 \le u \le n$. Let $A \subset \Z^n$.
Let $s \in \N^*$. If
$A$ is an $s$-cold set for
$(\Z^n, c_u)$, then
for every index~$i$, $p_i(A)$ is an infinite 
set, with sequences of members
tending both to $\infty$ and to $-\infty$.
\end{thm}

\begin{proof}
Suppose otherwise. Then for some~$i$, 
there exist $m$ or $M$ in $\Z$ such that
\[ m = \min\{p_i(a) \, | \, a \in A\}~~~
   \mbox{or}~~~M = \max\{p_i(a) \, | \, a \in A\}.
\]

If the former, then
for $z = (z_1, \ldots, z_n) \in \Z^n$,
define $f: \Z^n \to \Z^n$ by
\[ f(z) = \left \{ \begin{array}{cl}
   (z_1, \ldots, z_{i-1}, m, z_{i+1}, \ldots, z_n)
      & \mbox{if } z_i \le m; \\
    z & \mbox{otherwise.}
\end{array} \right .
\]
Then $f \in C(\Z^n, c_u)$ and $f|_A = \id_A$,
but $f \neq \id_{\Z^n}$. Thus, $A$ 
is not an $s$-cold set.

Similarly, if $M < \infty$ as above exists, 
we conclude $A$ is not an $s$-cold set.
\end{proof}

\begin{cor}
$A \subset \Z$ is a freezing set for $(\Z,c_1)$
if and only if $A$ contains
sequences $\{a_i\}_{i=1}^{\infty}$ and
$\{a_i'\}_{i=1}^{\infty}$ such that
$\lim_{i \to \infty} a_i = \infty$ and
$\lim_{i \to \infty} a_i' = -\infty$.
\end{cor}

\begin{proof}
This follows from Lemma~\ref{c1pulling} and
Theorem~\ref{infiniteFreeze}.
\end{proof}

The converse of
Theorem~\ref{infiniteFreeze}
is not generally correct, as
shown by the following.

\begin{exl}
Let $A = \{(z,z) \, | \, z \in \Z\} \subset \Z^2$. Then
although $p_1(A)=p_2(A)=\Z$ contains sequences
tending to $\infty$ and to
$-\infty$, $A$ is not an
$s$-cold set for $(\Z^2,c_2)$,
for any~$s$.
\end{exl}

\begin{proof}
Consider $f: \Z^2 \to \Z^2$
defined by $f(x,y)=(x,x)$.
We have $f \in (\Z^2,c_2)$
and $f|_A = \id_A$, but one sees easily that for
all $s$ there exist $(x,y) \in \Z^2$ such that
$d((x,y),f(x,y)) > s$.
\end{proof}

\section{Further remarks}
We have continued the work of~\cite{bs19a} in 
studying fixed point invariants and related
ideas in digital topology.

We have introduced pointed versions of rigidity
and fixed point spectra.

We have introduced the
notions of freezing sets and $s$-cold sets. These
show us that although
knowledge of the fixed point set $\Fix(f)$ 
of a continuous self-map~$f$
on a connected topological space~$X$ generally
gives us little information about the nature of
$f|_{X \setminus \Fix(f)}$, if $f \in C(X,\kappa)$ and
$A \subset \Fix(f)$ is a freezing set or, more
generally, an $s$-cold set for $(X,\kappa)$,
then $f|_{X \setminus \Fix(f)}$ may be
severely limited.

\section{Acknowledgment}
P. Christopher Staecker and an anonymous reviewer were most helpful. They each suggested several
of our assertions, and several corrections.

https://www.sciencedirect.com/science/article/abs/pii/0167865594900124

\bibitem{Bx99}
 L. Boxer, A Classical Construction for the Digital Fundamental Group, Journal of Mathematical Imaging and Vision 10 (1999), 51-62.
 
 https://link.springer.com/article/10.1023/A

 \bibitem{BxNormal}
 L. Boxer, Generalized normal product adjacency in digital topology, 
 {\em Applied General Topology} 18 (2) (2017), 401-427

https://polipapers.upv.es/index.php/AGT/article/view/7798/8718

\bibitem{BxAlt}
L. Boxer, Alternate product adjacencies in digital topology,
{\em Applied General Topology} 19 (1) (2018), 21-53

https://polipapers.upv.es/index.php/AGT/article/view/7146/9777

\bibitem{BxVigo}
L. Boxer, 
Fixed points and freezing sets in digital topology, 
{\em Proceedings, Interdisciplinary Colloquium in Topology and its Applications} in Vigo, Spain; 55-61.

\bibitem{BEKLL}
 L. Boxer, O. Ege, I. Karaca, J. Lopez, and J. Louwsma, Digital fixed points, approximate fixed points, and universal functions, 
 {\em Applied General Topology}
 17(2), 2016, 159-172.
 
 https://polipapers.upv.es/index.php/AGT/article/view/4704/6675
 
 \bibitem{BK12}
 L. Boxer and I. Karaca,
 Fundamental groups for digital products, 
 {\em Advances and Applications in Mathematical Sciences} 11(4) (2012), 161-180.
 
 http://purple.niagara.edu/boxer/res/papers/12aams.pdf
 
\bibitem{bs16}
L. Boxer and P.C. Staecker, Fundamental Groups and Euler Characteristics of Sphere-like Digital Images, 
{\em Applied General Topology} 17(2), 2016, 139-158. 

https://polipapers.upv.es/index.php/AGT/article/view/4624/6671

\bibitem{bs19a}
L. Boxer and P.C. Staecker,
Fixed point sets in digital topology, 1,
{\em Applied General Topology}, to appear.

https://arxiv.org/pdf/1901.11093.pdf

\bibitem{CandL}
G. Chartrand and L. Lesniak, {\em Graphs \& Digraphs}, 
2nd ed., Wadsworth, Inc., 
Belmont, CA, 1986.

\bibitem{hmps} 
J. Haarmann, M.P. Murphy, C.S. Peters, and P.C. Staecker,
Homotopy equivalence in finite digital images,
{\em Journal of Mathematical Imaging and Vision}
53 (2015), 288-302.

https://link.springer.com/article/10.1007/s10851-015-0578-8

\bibitem{Han05}
S-E Han, Non-product property of the digital fundamental group, 
{\em Information Sciences} 171 (2005), 7391.

https://www.sciencedirect.com/science/article/pii/S0020025504001008

\bibitem{Khalimsky}
E. Khalimsky,
 Motion, deformation, and homotopy in finite spaces, in
{\em Proceedings IEEE Intl. Conf. on Systems, Man, and Cybernetics},
1987, 227-234.

\bibitem{RosenfeldMAA}
A. Rosenfeld,
Digital topology,
{\em The American Mathematical Monthly} 86 (8) (1979), 621-630.

\bibitem{Rosenfeld}
A. Rosenfeld, `Continuous' functions on digital pictures, {\em Pattern Recognition Letters} 4,
pp. 177-184, 1986.

https://www.sciencedirect.com/science/article/pii/0167865586900176

\end{document}